\documentclass[11pt]{amsart}

\usepackage[english]{babel}

\usepackage{epigamath}
\usepackage{microtype}

\numberwithin{equation}{section}

\usepackage{xcolor}

\makeatletter
\newtheorem {rep@theorem}{\rep@title}
\newcommand{\newreptheorem}[2]{%
\newenvironment{rep#1}[1]{%
 \def\rep@title{#2 \ref{##1}}%
 \begin{rep@theorem}}%
 {\end{rep@theorem}}}
\makeatother

\newtheorem{theorem}{Theorem}[section]

\newtheorem{proposition}[theorem]{Proposition}

\theoremstyle{definition}

\newtheorem{remark}[theorem]{Remark}

\newreptheorem{Thm}{Theorem}

\newreptheorem{Cor}{Corollary}

\newreptheorem{Con}{Conjecture}

\newtheorem{thm-int}{Theorem}

\theoremstyle{definition}
\newtheorem{Def-s}[Thm]{Definition}

\newcommand{\Zz}{\mathbb{Z}}
\newcommand{\Cc}{\mathbb{C}}
\newcommand{\Pp}{\mathbb{P}}
\newcommand{\Rr}{\mathbb{R}}
\newcommand{\Qq}{\mathbb{Q}}

\newcommand{\Ker}{\operatorname{Ker}}

\newcommand{\ii}{\mathrm{i}}

\newcommand{\Bb}{\mathcal{B}}

\newcommand{\Oo}{\mathcal{O}}

\def\OO{\ensuremath{\mathcal O}}

\def\dim{\mathop{\mathrm{dim}}\nolimits}

\def\Ker{\mathop{\mathrm{Ker}}\nolimits}

\def\Proj{\mathop{\mathrm{Proj}}}

\def\Sym{\mathop{\mathrm{Sym}}\nolimits}

\EpigaVolumeYear{4}{2020} \EpigaArticleNr{10} \ReceivedOn{August 1,
2019}
\InFinalFormOn{December 20, 2019}
\AcceptedOn{June 4, 2020}

\title{Explicit equations of the Cartwright-Steger surface}
\titlemark{Equations of the CS surface}

\author{Lev A. Borisov}
\address{Department of Mathematics, Rutgers University, Piscataway, NJ 08854} \email{borisov@math.rutgers.edu}
\author{Sai-Kee Yeung}
\address{Department of Mathematics, Purdue University, West Lafayette, IN 47906} \email{yeung@math.purdue.edu}

\authormark{L. A. Borisov and S.-K. Yeung}

\AbstractInEnglish{We construct explicit equations of Cartwright-Steger and related surfaces.}

\MSCclass{14J29; 14Q10}

\KeyWords{Surfaces of general type; explicit equations; ball quotients}

\TitleInFrench{\'Equations explicites de la surface de Cartwright-Steger}

\AbstractInFrench{Nous construisons des \'equations explicites pour la surface de Cartwright-Steger et pour d'autres surfaces similaires.}

\acknowledgement{L.B. was partially supported by the NSF grant DMS-1601907. S.-K. Y. was partially supported by the NSF grant DMS-1501282.}

\begin{document}



\maketitle

\begin{prelims}

\DisplayAbstractInEnglish

\bigskip

\DisplayKeyWords

\medskip

\DisplayMSCclass

\bigskip

\languagesection{Fran\c{c}ais}

\bigskip

\DisplayTitleInFrench

\medskip

\DisplayAbstractInFrench

\end{prelims}


\newpage

\setcounter{tocdepth}{2}

\tableofcontents


\section{Introduction}
The Cartwright-Steger surface is a smooth projective algebraic surface of Euler number $3$.  It is a complex two ball quotient with first Betti number $2$.
It was discovered by Cartwright and Steger in their work \cite{CS} in the process of completing the scheme of classification of fake projective planes proposed in
\cite{PY}.   Basic properties of the Cartwright-Steger surface can be found in \cite{CS} and \cite{CKY}.

\smallskip
Our goal in this paper is to give explicit defining equations for a Cartwright-Steger surface  $Z$.  This is motivated in part by the recent work of \cite{BK}, where explicit
equations were given for two fake projective planes.  

\smallskip
There are several corollaries of theoretical interest to our main result.  
 One corollary is that the Cartwright-Steger surface can be defined over $\Rr$ and even over $\Qq$.  Hence
 the complex conjugate of the surface is biholomorphic to itself.  This corrects a mistake in \cite{Y2} which we explain in Remark \ref{mistake}.
 
\smallskip
To put our work in a more global perspective, it is well known that $3$ is the smallest Euler number achievable by a smooth surface of general type and is achieved only 
 if the surface is a smooth
 complex two ball quotient $\Bb^2/\Gamma$ for some torsion free lattice $\Gamma$ in $PU(2,1)$, the automorphism group of the complex two ball $\Bb^2$.  The lattice has to be arithmetic following the work of
 \cite{Y1}, \cite{Y2} or \cite{EG}, where the recent paper \cite{EG} gives a general result about integrality of lattices and is applicable to our surface.  The results of \cite{PY} and \cite{CS} then imply that such a surface is either a
 fake projective plane or a Cartwright-Steger surface.  Together with the result in this paper, there are exactly $101$ such surfaces.  Cartwright-Steger surface
 is the only one with positive first Betti number.

\smallskip 
Another consequence of this paper is that we show explicitly that the bicanonical divisor $2K_Z$ is  very ample and the  Cartwright-Steger surface $Z$ is defined by a quadratic relation and cubic relations  among global sections of $2K_Z$. 
This 
 improves the result of \cite{Y3} for $Z$ and fills a gap there, see Remark \ref{gap}.

\smallskip
In addition to explicit equations for $Z$, we have also obtained explicit equations and geometric information for a few related interesting surfaces
described below.

\smallskip
Denote by $\Pi$ the lattice associated to $Z$ so that $Z=\Bb^2/\Pi$.  Let $\bar\Gamma$ be the maximal arithmetic lattice in the commensurable class of $\Pi$ as
defined in \cite[{\S~1.1}]{CKY}. It was observed by Cartwright and Steger that $\Pi$ can be written as the intersection 
$$
\Pi = \Pi_2\cap \Pi_3
$$ 
where
\begin{itemize}
\item
$\Pi_2$ is an index $3$ normal subgroup of $\bar\Gamma$;
\item
$\Pi_3$ contains the kernel of ${\rho_3} :\bar\Gamma\rightarrow PGU(3,3)$ 
and is the preimage of a subgroup $G_{21}$ of  $PGU(3,3)$;
\end{itemize} 
see \cite[\S~1.1]{CKY}.
We consider the following surfaces.
\begin{itemize}
\item
the Cartwright-Steger surface $Z=\Bb^2/\Pi$;
\item
the $21:1$ cover of $Z$ given by $Z_2=\Bb^2/{\Ker(\rho_3) \cap \Pi_2}$;
\item
the quotient of $Z_2$ by $C_3$ given by $\Bb^2/\Ker(\rho_3)$ and its minimal resolution
$Z_3=\widehat{\Bb^2/\Ker(\rho_3)}$;
\item
the quotient $Z/C_3=\Bb^2/\Pi_3$ and its partial resolution
$Z_1=\widehat{Z/C_3}$.
\end{itemize}
These surfaces fit in the commutative diagram
\begin{equation}\label{blueprint}
\begin{array}{ccccc}
              Z_2&  \longrightarrow         &   Z_2/C_3      &   \longleftarrow    & Z_3  
\\
 \downarrow&                               &  \downarrow   &                           &\downarrow
\\
              Z &   \longrightarrow       &Z /C_3          &      \longleftarrow         &      Z_1
\end{array}
\end{equation}
where the vertical maps are quotients by a group of order $21$, the maps $\longrightarrow$ are quotients by 
a cyclic group $C_3$, and the maps $\longleftarrow$ are birational morphisms.

\smallskip
The structure of the paper is as follows. In Section \ref{sec.Z3} we consider the surface $Z_3$ which we describe explicitly by $35$ quadratic equations in $\Cc\Pp^{12}$.
In Section \ref{sec.CSC3} we use the results of Section \ref{sec.Z3} to construct explicit equations of the surface $Z_1$. In Section \ref{sec.Z2} we describe a procedure that we used to \emph{guess} equations of the surface $Z_2$ and specifically construct a number of points on it with high accuracy. While we do not prove that they are correct, we expect them to be. In Section \ref{sec.Z} we describe the method we used to calculate the putative equations of $Z$. Finally, in Section \ref{sec.verify} we explain the collection of computer-based checks that ensure that the equations we found indeed describe the Cartwright-Steger surface. The  ancillary files
for computations are available on arXiv:1804.00737.

\medskip
{\bf Acknowledgements.} L.B. thanks Igor Dolgachev, JongHae Keum and Carlos Rito for stimulating questions and useful comments. S.-K. Y. is grateful to Donald Cartwright for helpful discussions, and to Vincent Koziarz and
Gopal Prasad for their interest.
We used software packages GAP \cite{GAP}, Mathematica \cite{Math}, Magma \cite{magma} and Macaulay2 \cite{M2}.

\section{Constructing the surface $Z_3=\widehat{\Bb^2/\Ker(\rho_3)}$}\label{sec.Z3}
In this section we describe the equations of the surface $Z_3$ in \eqref{blueprint}, which is the key starting point of our method.

\smallskip
Let $\rho_3:\bar\Gamma \to PGU(3,3)$ be the reduction homomorphism modulo three, described in \cite{CKY}.
The action of $\Gamma_3 = \Ker(\rho_3)$ on $\Bb^2$ has fixed points, and the quotient has $63$ singular points of type $(\frac 13,\frac 13)$. 
The minimal resolution of singularities is a surface $Z_3$ with $K_{Z_3}^2 = 42$, $h^{1,0}(Z_3)=0$ and $h^{2,0}(Z_3)=13$. 
Importantly, $Z_3$ admits an effective action of the finite simple group $G=PGU(3,3)$ of order $6048$. The character table of $G$ which is very important for our 
discussion can be readily obtained with GAP software package \cite{GAP} and is presented in Table \ref{gaptable}. \footnote{The table presents values of characters of irreducible representations of $G$ at all conjugacy classes of $G$. Conjugacy classes are labeled by their order and a letter to distinguish different conjugacy classes of elements of the same order. Zero values are replaced by $\cdot$ for ease of reading.}

\begin{table}[tbh]
\caption{Character table of $G=PGU(3,3)$.}
$$
\begin{array}{|lrrrrrrrrrrrrrr|}\\[-3em]
\hline
    &    1a & 7a &7b &2a &3a &4a &4b &6a &12a &12b& 8a& 8b& 4c& 3b \\[.5em]
 \chi_{1}   &    1 &  1&   1 &  1 &  1 &  1 &  1 &  1 &   1 &   1 &  1 &  1 &  1 &  1  \\
 \chi_{2}    &    6 & -1 & -1&  -2&  -3 & -2&  -2  & 1 &   1  &  1 &  . &  . &  2 &  . 		\\
 \chi_{3}     &   7  & . &  . & -1 & -2  & 3 &  3  & 2  &  . &   . & -1 & -1&  -1 &  1		\\
 \chi_{4 }    &   7  & . &  .  & 3 & -2  & B & \overline B  & .  &  E &  \overline E &  F&  \overline F &  1 &  1		\\
 \chi_{5   } &    7  & . &  .  & 3 & -2 & \overline B &  B  & . &  \overline E &   E & \overline F &  F &  1 &  1		\\
 \chi_{6  }  &   14 &  . &  . & -2  & 5  & 2 &  2  & 1 &  -1 &  -1 &  .&   .&   2&  -1		\\
 \chi_{7  }  &   21 &  . &  . &  5  & 3  & 1  & 1 & -1  &  1 &   1 & -1 & -1 &  1 &  .		\\
 \chi_{8  }  &   21 &  .  & .  & 1  & 3  & C & \overline C &  1 &   F &  \overline F & \overline F &  F & -1 &  .	\\
 \chi_{9  }  &   21 &  .  & .  & 1  & 3 & \overline C  & C  & 1 &  \overline F &   F &  F&  \overline F&  -1 &  .	\\
 \chi_{10}   &   27&  -1 & -1 &  3&   .&   3 &  3  & . &   . &   . &  1 &  1 & -1 &  .		\\
 \chi_{11}   &   28 &  .  & . & -4  & 1&   D & \overline D & -1 &  \overline F &   F &  .&   .&   . &  1		\\
 \chi_{12 }  &   28 &  .  & . & -4  & 1&  \overline D &  D & -1  &  F &  \overline F &  . &  . &  . &  1		\\
 \chi_{13}   &   32 &  A & \overline A & . & -4 &  . &  . &  . &   . &   . &  . &  . &  . & -1		\\
 \chi_{14}   &  32 & \overline A  & A  & . & -4 &  . &  . &  .  &  .  &  .  & .  & . &  . & -1	
 \\\hline	
\end{array}
$$
$$
(A,B,C,D,E,F) =(\frac {1 -\ii\sqrt 7}2 ,-1-2\ii, -3-2\ii, -4\ii, -1-\ii, -\ii).
$$
\label{gaptable}
\end{table}

\smallskip
There is a unique up to conjugation noncommutative subgroup $G_{21}\subset G$ which normalizes a  $7$-Sylow subgroup of $G$. As we pointed out in \eqref{blueprint},  the quotient of $Z_3$ by this subgroup $G_{21}$ is the blowup $Z_1$ of $Z/C_3$ in three singular points $O_1, O_2$ and $O_3$ of type $(\frac 13, \frac 13)$. Observe that $Z_1$ has Gorenstein singularities and its canonical class has a global section. Specifically, in the notation of \cite{CKY}
there is a (canonical) curve $E_3$ on $Z$ which gives a curve $D_3:=E_3/C_3$ on $Z_3/C_3$.  From \cite{CKY}, $E_3$ passes once through $O_1$ and $O_3$, four times through $O_2$ and does not pass through any of the six other fixed points. 
Denote by $\widehat{D}_3$ the proper transform of $D_3$ and $G_i$ for $i=1,2,3$ the $(-3)$ exceptional curve at $O_i$ of $Z_1\to Z/C_3$.  It follows easily 
that
$K_{Z_1}=\pi^*K_S-\frac13\sum_{i=1}^3 G_i$ and $K_{Z_1}=\widehat{D}_3+G_2$.
Hence a global section of $K_{Z_1}$ exists and is consisting of the proper preimage of $D_3$ plus one $(-3)$ exceptional curve $G_2$.
This global section pulls back on $Z_3$ to give an element $h\in H^0(Z_3,K_{Z_3})$ which is $G_{21}$-invariant, but not $G$-invariant. We will denote the corresponding canonical divisor on $Z_3$ by $H$.

\smallskip
We will now be able to identify the action of $G$ on $H^0(Z_3,K_{Z_3})$. By a slight abuse of terminology, we will denote by $\chi_k$ both the $k$-th irreducible representation of $G$ and its character.

\begin{proposition}\label{35eqs}
Up to a choice of square root of $(-1)$, the action of $G$ on $H^0(Z_3,K_{Z_3})$ is given by the direct sum of irreducible representations $\chi_2 +\chi_4$.
The equations on $\Sym^2(\chi_2 + \chi_4)$ are given as follows. We have 
$\Sym^2(\chi_2) = \chi_{7}$ and $\Sym^2(\chi_4)= \chi_5+\chi_7$. We also have $\chi_2\chi_4=\chi_{6}+\chi_{12}$.
There is a $35$-dimensional space of quadratic relations which is built from the $14$-dimensional irreducible subspace $\chi_6$ of $\chi_2\chi_4$
and the $21$-dimensional subspace that lies diagonally in the direct sum of two copies of $\chi_{7}$. 
\end{proposition}

\begin{proof}
We know that the action of the order seven elements of $G$ on $Z_3$ is fixed-point free. By the Holomorphic Lefschetz Formula (since $H^2(Z_3,K_{Z_3})$ is acted on trivially) we 
see that the trace of these elements on $H^0(Z_3,K_{Z_3})$ is $(-1)$. Together with $\dim H^0(Z_3,K_{Z_3}) =13$ we see that the possible characters are 
$$
\chi_{H^0(Z_3,K_{Z_3})} \in \{ \chi_1 + 2\chi_2, \chi_2 + \chi_3,\chi_2+\chi_4,\chi_2+\chi_5\}.
$$
{\bf Case 1.}  $\chi_{H^0(Z_3,K_{Z_3})} = \chi_1 + 2\chi_2$. This is impossible, because then the $G_{21}$-invariant canonical divisor $H$ would be $G$-invariant.

\smallskip\noindent
{\bf Case 2.}  $\chi_{H^0(Z_3,K_{Z_3})} = \chi_2 + \chi_3$. We have used GAP to readily calculate
$$
\Sym^2(\chi_2+\chi_3) = \Sym^2(\chi_2) + (\chi_2 \chi_3)  + \Sym^2(\chi_3) = \chi_7 + (\chi_8+\chi_9) + (\chi_1 +\chi_{10}).
$$ 
We know that the kernel of the map 
$$
\Sym^2(H^0(Z_3,K_{Z_3}) )\to H^0(Z_3,2K_{Z_3})
$$
between spaces of dimension $91$ and $56$ respectively is of dimension at least $35$. However, $\chi_7$ component of $\Sym^2(H^0(Z_3,K_{Z_3}))$  cannot be in the kernel.
Indeed, in that case, the quadratic equations on the six-dimensional subspace of $H^0(Z_3,K_{Z_3})$ would imply that they are zero.
Similarly, $\chi_{10}$ component cannot go to zero, since $27$-dimensional space of quadratic equations on $7$ variables implies that they are zero. This means that both $\chi_8$ and $\chi_9$ components must vanish, which again leads to a contradiction, since this would mean that the product of two nonzero sections of $H^0(Z_3,K_{Z_3})$ (one from $\chi_2$ and another from $\chi_3$) would be zero.

\smallskip\noindent
{\bf Case 3.}  $\chi_{H^0(Z_3,K_{Z_3})} = \chi_2 + \chi_4$ or $\chi_{H^0(Z_3,K_{Z_3})} = \chi_2 + \chi_5$. We may assume the former case by either taking complex conjugate of alternatively applying an outer automorphism of $G$. We have 
$$
\Sym^2(\chi_2+\chi_4) = \Sym^2(\chi_2) + (\chi_2 \chi_4)  + \Sym^2(\chi_4) = \chi_7 + (\chi_6+\chi_{12}) + (\chi_5 +\chi_{7})
$$ 
and consider the kernel of the map 
$$
\Sym^2(H^0(Z_3,K_{Z_3}) )\to H^0(Z_3,2K_{Z_3}).
$$
As in Case 2, we see that either the kernel of the map $\Sym^2(H^0(Z_3,K_{Z_3}) )\to H^0(Z_3,2K_{Z_3})$ has no $\chi_7$ parts or at most one part. 
By looking at the $\chi_2 \chi_4=\chi_6+\chi_{12}$ component we see that $\chi_{12}$  cannot go to zero. Indeed, Magma \cite{magma} calculations show that the scheme cut out by
this $28$-dimensional space of relations has the same Hilbert polynomial as the union of ${\bf x}={\bf 0}$ and ${\bf y}={\bf 0}$. This means that these relations have no nontrivial solutions.
Together with the fact that dimension of the kernel is at least $35$ we see that the kernel must contain one copy of $\chi_7$ and $\chi_6$. If this copy of $\chi_7$ is not diagonal, then Magma calculations show that the corresponding sections are zero; this means that $\chi_7$ must lie diagonally. Finally, we see that if, in addition to the diagonal copy of $\chi_7$ we have $\chi_5$ in the kernel, then the scheme cut out by these equations is empty (we did these calculations modulo $(10-\ii)$, which is sufficient).
\end{proof}

We denote the variables which correspond to the basis of $\chi_2$ and $\chi_4$ by $x_1,\ldots,x_6$ and $y_1,\ldots,y_6$ respectively.
The action of the two standard generators of the group $G$ on these variables is given by Table \ref{table.action}. We used the data from the Atlas of Finite Group Representations \cite{A}.
\begin{table}[tbh]
\caption{Action of $G$ on $(x_1,\ldots, x_6,y_1,\ldots,y_7)$}
$$
\begin{array}{|ll|}
\\[-3em]
\hline
(x_1,\ldots,y_7)&\mapsto (x_2, x_1, x_4, x_3, -x_5, (-2 + \ii) x_1 - (2 - \ii) x_2 + \ii x_3 + \ii x_4 - \ii x_6,
\\
& y_2, y_1, y_4, y_3, (1 + \ii) y_1 - (1 + \ii) y_2 - y_3 + y_4 + y_5, y_6, 
  \\
 &
 \ii y_3 - \ii y_4 + y_7)
 \\[.5em]
(x_1,\ldots,y_7)&\mapsto (x_3, -\ii x_1 - x_3, x_5, x_6, -\ii x_1 - \ii x_2 - x_3, 
 x_1 + (1 - \ii) x_2 
 - \ii x_3 
  \\
 &
 - x_4 - x_5 + x_6,
 y_3, y_1, y_5, y_6, y_7, -\ii y_2 - \ii y_3 - y_4 - y_6 - 
 \ii y_7, 
 \\
 &
 -y_1 + y_2 + y_3 - y_5 + y_7)
 \\
 \hline
\end{array}
$$
\label{table.action}
\end{table}

In what follows we will use a specific choice of $G_{21}\subset G$ whose generators act on $x_i$ and $y_j$ by the formulas of Table \ref{G21XY}.
\begin{table}[tbh]
\caption{Action of $G_{21}$ on $(x_1,\ldots, x_6,y_1,\ldots,y_7)$}
$$
\begin{array}{|ll|}
\\[-3em]
\hline
(x_1,\ldots,y_7)&\mapsto (-\ii x_1 - x_3, x_3, x_6, x_5, \ii x_1 + \ii x_2 + x_3, 
 2 \ii x_1 - (1 - \ii) x_2 + \ii x_3 
 \\
 &
 + x_4 + (1 + \ii) x_5 - (1 - \ii) x_6,
y_1, y_3, y_6, y_5, (-1 - \ii) y_1 + (1 + \ii) y_3
\\
 &
  - y_5 + y_6 + y_7, -\ii y_2 - 
  \ii y_3 - y_4 - y_6 - \ii y_7, -y_1 + y_2 + y_3 
  \\
 &
 - (1 - \ii) y_5 - \ii y_6 + y_7)
 \\[.5em](x_1,\ldots,y_7)&\mapsto
((-1 - \ii) x_1 - x_2 - (1 - \ii) x_3 - \ii x_4 + x_5 - x_6, 
 \ii x_1 + (1 + \ii) x_2 
 \\
 &
 + x_3 - \ii x_5, -x_1 - \ii x_2 - \ii x_3 - (1 - \ii) x_4 - 
  2 x_5, (1 + \ii) x_1 + x_2
  \\
 &
  + x_3 + \ii x_6, -\ii x_2 - (1 + \ii) x_3 + 
  \ii x_4 - (1 - \ii) x_5 - \ii x_6, 2 x_1 + x_2
  \\
 &
  + (1 - \ii) x_3 + x_6,
y_1, \ii y_1 - \ii y_2 - 2 \ii y_3 - y_4 + (1 + \ii) y_5 - y_6
\\
 &
  - (1 + 2 \ii) y_7, -y_3 +
   \ii y_4 + \ii y_6 - (1 - \ii) y_7, -y_1 + (1 - \ii) y_2 - 2 \ii y_3 
   \\
 &
 - y_4 + \ii y_5 - 
  y_6 - \ii y_7, y_4, (1 + \ii) y_1 - (1 + \ii) y_2 - y_3 + y_4 + 
  y_5, 
  \\
 &
 (1 + \ii) y_2 + (2 + \ii) y_3 + (2 - \ii) y_4 - 
  y_5 + (1 - 2 \ii) y_6 + (2 + \ii) y_7)
 \\  \hline
\end{array}
$$
\label{G21XY}
\end{table}

\begin{remark}\label{Hequation}
For the choice of $G_{21}$ in Table \ref{G21XY}, the $G_{21}$-invariant section of $K_{Z_3}$ which gives $H$ is readily computed to equal
$$
7 y_1 + y_2 + y_3 - (2 + 3 \ii) y_4 - (2 + 3 \ii) y_5 + y_6 + y_7
$$
up to scaling. 
\end{remark}

\begin{remark}
While it is impractical to write down all equations in the kernel of
$$\Sym^2(H^0(Z_3,K_{Z_3}) )\to H^0(Z_3,2K_{Z_3}),$$ we write two of them in Table \ref{Z3equations} that generate the diagonal $\chi_7$ and $\chi_6$ respectively. The rest can be obtained by applying the group action. We note that one can scale the variables $x$ and $y$ separately in the diagonal equation, but we used a particular choice that ensured relatively simple formulas.
\end{remark}

\begin{table}[tbh]
\caption{Equations of  $Z_3=\widehat{\Bb^2/\Ker(\rho_3)}$}
$$\small
\begin{array}{|l|}
\\[-3em]
\hline
x_1^2 + 2 (4 y_1^2 + (2  \ii ) y_1 y_2 - (2 - 2  \ii ) y_2^2 - 2 y_3^2 + (3 - 5  \ii ) y_3 y_4 + 
    (1 - 2  \ii ) y_4^2 - (2 - 2  \ii ) y_3 y_5
     \\
    + (3 + 3  \ii ) y_4 y_5 - (2 - 2  \ii ) y_5^2 +  (1 + 5  \ii ) y_3 y_6 - 12 y_4 y_6     - (1 + 5  \ii ) y_5 y_6 + (1 + 2  \ii ) y_6^2
    \\ + 
    (1 + 5  \ii ) y_3 y_7 + 4 y_4 y_7 - (1 - 3  \ii ) y_5 y_7 + (2 - 12  \ii ) y_6 y_7 + 
    (1 + 2  \ii ) y_7^2 + (1 -  \ii ) y_1 ((1 +  \ii ) y_3 
    \\
    + (1 - 4  \ii ) y_4 + (1 - 3  \ii ) y_5 + 
      y_6 + y_7) + (1 -  \ii ) y_2 ((5 +  \ii ) y_3 + 3 y_4 + (3 -  \ii ) y_5 - (6 + 7  \ii ) y_6
      \\
       + 
      (2 - 3  \ii ) y_7)) =0 \\[.5em] 
       (1 -  \ii ) ((-10 - 6  \ii ) x_3 y_1 - (4 + 10  \ii ) x_4 y_1 + (4 - 3  \ii ) x_5 y_1 + 
   (4 +  \ii ) x_6 y_1 + (5 - 9  \ii ) x_3 y_2 
   \\
   + (2 - 4  \ii ) x_4 y_2 + x_5 y_2 - 
   (2 + 2  \ii ) x_6 y_2 - (10 + 9  \ii ) x_3 y_3 - (1 +  \ii ) x_4 y_3 + (1 + 6  \ii ) x_5 y_3
   \\
    + 
   (1 + 10  \ii ) x_6 y_3 + (2 - 6  \ii ) x_3 y_4 - (1 - 5  \ii ) x_4 y_4 - 8 x_5 y_4 + 
   (1 + 7  \ii ) x_6 y_4 + (2 + 9  \ii ) x_3 y_5 
   \\
   - (4 - 2  \ii ) x_4 y_5 + (1 + 3  \ii ) x_5 y_5 - 
   (5 + 2  \ii ) x_6 y_5 - (10 - 3  \ii ) x_3 y_6 + (2 + 2  \ii ) x_4 y_6 + 4 x_5 y_6
   \\
    + 
   (13 - 5  \ii ) x_6 y_6 + x_1 ((3 + 2  \ii ) y_1 + (3 + 2  \ii ) y_2 - 
      \ii  ((13 + 3  \ii ) y_3 + (1 + 3  \ii ) y_4 + y_5
      \\
       + (1 - 12  \ii ) y_6 - (2 + 3  \ii ) y_7)) + 
   x_2 ((13 - 2  \ii ) y_1 + (1 + 7  \ii ) y_2 + (7 + 7  \ii ) y_3 - (2 + 2  \ii ) y_4 
   \\
   - 
     (5 + 2  \ii ) y_5 + (7 - 2  \ii ) y_6 - (2 + 5  \ii ) y_7) + (5 + 9  \ii ) x_3 y_7 - 
   (4 - 2  \ii ) x_4 y_7 + (1 - 6  \ii ) x_5 y_7 
   \\
   - (2 -  \ii ) x_6 y_7) =0 
   \\
   \hline
   \end{array}
$$
\label{Z3equations}
\end{table}

We will eventually prove that these $35$ quadratic relations cut out a surface in $\Cc\Pp^{12}$ which is isomorphic to $Z_3$. However, it is a fairly delicate argument, partly because some of the computer calculations that could have simplified it are too time consuming to perform with available hardware and software. 

\smallskip
We start by getting a rough idea of what these equations cut out.
\begin{proposition}\label{Srough}
Let $I$ be the ideal of $\Cc[x_1,\ldots,y_7]$ generated by the dimension $35$ space of relations from $\chi_7$ and $\chi_6$, which are the $G$-translates of Table \ref{Z3equations}. 
Then the dimension of the degree $k$ component of $\Cc[x_1,\ldots,y_7]/I$ is equal to $21k(k-1)+14$ for all $k\geq 2$. The scheme $S = \Proj \Cc[x_1,\ldots,y_7]/I$ 
is a disjoint union of smooth surfaces of total degree $42$ and perhaps some curves and points.
\end{proposition}

\begin{proof}
These results are the consequence of computations we performed with Magma \cite{magma}. The Hilbert polynomial calculation was quickly performed with coefficients in $\Qq[\ii]$.
We can also show by Magma that these relations and the $10\times 10$ minors of the Jacobian matrix of their partial derivatives have no common solutions. For this we picked some minors only, otherwise there are too many of them. We worked modulo prime number $(10-\ii)$ in $\Zz[\ii]$, which is sufficient from semicontinuity. 

The Hilbert polynomial computation implies that $S$ has no irreducible components of dimension more than $2$, and the sum of the degrees of surface components is $42$. 
The calculation of minors implies that surface components are smooth and do not intersect any other components.
\end{proof}

\begin{remark}\label{rem.involution}
The scheme $S$ has an additional involution given by 
$$(x_1,\ldots,x_6,y_1,\ldots,y_7)\mapsto (-x_1,\ldots,-x_6,y_1,\ldots,y_7)$$
which commutes with the action of $G$. 
\end{remark}

\begin{remark}
A Mathematica \cite{Math} calculation shows that scheme $S\subset \Cc\Pp^{12}$ contains a certain set of $126$ lines. 
These form two orbits with respect to the action of $G$ and one orbit with 
respect to the action of $G$ together with the involution of Remark \ref{rem.involution}. One of these lines is given by
$$
\begin{array}{rl}
(x_1,\ldots, y_7) = &(v, u, \ii v + (1 - \ii)u, (-1 - 2\ii)v + 2\ii u, (2 - \ii)v - 2u, 
  \\[.2em]&
  (-2 + 2\ii)v + 3u,
  (2 + 2\ii)v - (1 + 3\ii)u, (-1 - \ii)v + \frac 12(1+3\ii) u,
  \\[.2em]&
 \frac 12 (-3 - \ii)v + \frac 12(3 + 3\ii)u, \frac 12(1 - \ii)v - u,  2v - (2 + \ii)u,
 \\[.2em]&
 \frac 12 (-3 + \ii)v + \frac 12(3 + \ii)u, \frac 12(-1 - \ii)v + \frac 12(1 + \ii)u)
 \end{array}
$$
for $(u:v)\in \Cc\Pp^1$. The rest can be calculated by applying the symmetries.
We also observe that these lines intersect each other and form a connected set.
\end{remark}

\begin{proposition}\label{S0}
The scheme $S$ has a unique dimension two component $S_0$ of degree $42$.
\end{proposition}

\begin{proof}
Magma calculation shows that the Hilbert polynomial of the subscheme of $\Pp^{12}$ cut out by $I$ and the equation of Remark \ref{Hequation} is 
$42(k-1)$. On the other hand, we can check that $42$ out of $126$ lines lie in $S\cap H$. This means that these lines are the intersections of $H$ with two-dimensional components of $S$. Because these are a part of a connected set of lines, this means that they come from the same component $S_0$ which must then have degree $42$.
\end{proof}

\begin{proposition}
The $35$ quadratic equations from the kernel of $$\Sym^2(H^0(Z_3,K_{Z_3}) )\to H^0(Z_3,2K_{Z_3})$$ define an embedding of $Z_3$ into $\Cc\Pp^{12}$.
\end{proposition}

\begin{proof}
The base locus of $H^0(Z_3,K_{Z_3})$ is contained in the intersection of all $G$-translates of the $G$-invariant canonical curve $H$. Since we know the description of $H$ as 
a union of $21$ preimages of $E_3/C_3$ and $21$ exceptional curves, we see that this base locus is empty. 
Thus, $H^0(Z_3,K_{Z_3})$  gives a map $Z_3\to \Pp^{12}$. Since $K_{Z_3}^2=42>0$, we see that the image of $Z_3$ must be a surface.
We know that it is a subscheme of the scheme $S\subset \Pp^{12}$ which is cut out by $35$ explicit equations we found in Proposition \ref{35eqs}. By Proposition \ref{S0} we see that this image must be a smooth surface $S_0$.

We claim that $K_{Z_3}$ is ample. Indeed, base point freeness and $K_{Z_3}^2>0$
imply that $Z_3$ is a minimal surface of general type. Its map to the
canonical model may only contract some trees of $\Cc\Pp^1$ which intersect $H$ and its translates
trivially. Such curves would embed into $Z_2/C_3$, with image disjoint from the ramification points of
$\Bb^2\to Z_2/C_3$. Since $\Cc\Pp^1$ is simply connected, this embedding lifts to an
 embedding of $\Cc\Pp^1$ into $\Bb^2$, which is impossible.
 
Since the degree of $S_0$ is the same as $K_{Z_3}^2$, we see that the map $Z_3\to S_0$ is birational. Together with the ampleness, we see that it is an isomorphism.
We now see that the homogeneous coordinate ring of $S_0$ and the
pluricanonical ring of $Z_3$ coincide for high enough degrees. Since the
former is a quotient of the homogeneous coordinate ring of $S$ which has the
same Hilbert polynomial  $21(k^2-k)+14$ as the pluricanonical ring, we see that $S_0=S$, i.e. the $35$ quadratic equations
cut out $S=S_0\cong Z_3$ scheme-theoretically.
\end{proof}

\begin{remark}
While we believe that the embedding $Z_3\to \Cc\Pp^{12}$ is projectively normal and
the $35$ equations cut it out ideal-theoretically (rather than only up to
saturation) we were unable to verify it, since the required computer
calculations are very time-intensive.
\end{remark}

\begin{remark}
It is important to understand, heuristically, what factors contributed to our ability to identify $Z_3$ based on the group representation data, with the goal of extending these  ideas to construction of related surfaces, notably Galois covers of fake projective planes.
Specifically, the fact that $[\bar\Gamma:\Gamma]$ is large meant that the group $G$ was relatively big compared to the square of the canonical class, so group theoretic constraints were more severe.
The fact that $\Bb^2/\Ker(\rho_3)$ had non-canonical quotient singularities led to a lower value of $K_{Z_3}^2$ than what would be predicted by
$[\bar\Gamma:\Ker(\rho_3)]$ alone, which lowered the degree of $Z_3$ and made lower degree equations more likely. 
On the other hand, $h^{0,1}>0$ could in principle be beneficial, as a larger value of $h^{0,2}$ could 
lead to more equations of lower degree. 
\end{remark}

\section{On the quotient of CS surface by the cyclic group of order three.}\label{sec.CSC3}
In this section we describe the quotient $Z_1$ of  $Z_3 = \widehat{\Bb^2/\Ker(\rho_3)}$ by a group of order 21, which is also the blowup of the quotient of CS surface $Z$ by its automorphism group $C_3$ at three singular points of type $(\frac 13,\frac 13)$. The procedure was to compute $G_{21}$-invariant polynomials in the homogeneous coordinates of $\Pp^{12}$ and find the ring of invariants of $\Cc[x_0,\ldots,y_7]/I$. While there is only one such invariant in degree $1$, up to scaling, given by the equation of Remark \ref{Hequation}, there is a dimension $4$ space of invariants of degree $2$ and a dimension $8$ space of invariants of degree $3$, and choices must be made carefully to get simpler equations.

\smallskip
We are not particularly interested in this surface, but its equations\footnote{These are putative equations, since we found them by constructing multiple points on $Z_3$ with high precision and looked for relations on values of the invariant polynomials at these points. However, it should be straightforward to derive these equations from those of $Z_3$.} are (at the moment) nicer than those of $Z$ and we list them in Table \ref{CSC3}. It is embedded into the weighted projective space $W\Pp(1,2,2,2,3,3,3,3)$ with homogeneous coordinates 
$(W_0:\ldots:W_7)$. The variables $W_2,\ldots,W_5$ are odd with respect to the involution of Remark \ref{rem.involution}, and the other four variables are even. It took a fair bit of work to find the appropriate basis of variables and of equations to have them of palatable length. Remarkably, these equations have rational coefficients, and it was a non-trivial exercise to find the ring generators\footnote{Actually, $W_i$ only generate the ring at high enough degrees; we are missing a degree four generator.} for which this can be achieved. We have verified by Magma that these equations cut out a ring whose degree $d$ part has dimension $(d^2-d+2)$ for large $d$, which is consistent with the expected dimension. Indeed, we have on $Z_3$ 
$$
\dim H^0(Z_3,dK_{Z_3}) = 21d(d-1) + 14,
$$
and Holomorphic Lefschetz Formula allows us to find the dimensions of the subspaces of $G_{21}$-invariants.

\begin{table}[tbh]
\caption{Equations of $Z_1$.}
$$
\tiny\begin{array}{|rl|}
\\[-3em]
\hline
0=&-2W_1^2 + W_2W_3 + W_0W_7, 
 \\[.5em]
  0=&16W_0^5 + 152W_0^3W_1 + 21W_0W_1^2 + 2W_0W_2W_3 -
  W_0W_3^2 - 16W_2W_5 
  + 432W_0^2W_6 + 54W_1W_6,\\[.5em]
0=& 40W_0^3W_1 - 20W_0W_1^2 + 4W_0W_2^2 + 10W_0W_2W_3 - 108W_3W_4 +
  432W_0^2W_6 + 216W_1W_6 - W_1W_7,
  \\[.5em]
  0=&
   8W_0^3W_2 + 68W_0W_1W_2 - 6W_0W_1W_3 -
  432W_0^2W_4 
  - 216W_1W_4 - 512W_0^2W_5 - 64W_1W_5  \\[.5em]
  &+ 216W_2W_6 + 108W_3W_6 -
  W_2W_7,
   \\[.5em]
  0=&
   -1280W_0^6 - 10624W_0^4W_1 + 7776W_0^2W_1^2 + 320W_1^3 -
  80W_1W_2^2 + 80W_0^2W_3^2 - 48W_1W_3^2 
  \\[.5em]
  &
  + 3456W_0W_2W_4 - 46656W_4^2 +
  512W_0W_3W_5 - 4096W_5^2 - 34560W_0^3W_6 - 1728W_0W_1W_6 
    \\[.5em]
  &+ 128W_0^3W_7 +
  W_7^2, 
    \\[.5em]
  0=&224W_0^4W_1 + 976W_0^2W_1^2 + 48W_1^3 - 18W_1W_2^2 - 2W_1W_3^2 +
  1080W_0W_2W_4 
    - 11664W_4^2\\[.5em]
  & + 128W_0W_3W_5 - 3456W_0W_1W_6 - 8W_0^3W_7 +
  17W_0W_1W_7 + 54W_6W_7,
   \\[.5em]
  0=& 64W_0^6 + 640W_0^4W_1 + 209W_0^2W_1^2 + 8W_1^3 -
  W_0^2W_2^2 - 4W_1W_2^2 - 4W_0^2W_3^2 + 324W_0W_2W_4 
    \\[.5em]
  &- 2916W_4^2 +
  1728W_0^3W_6 - 108W_0W_1W_6 + 2916W_6^2 - 8W_0^3W_7 + 2W_0W_1W_7, \\[.5em]
  0=&
 -160W_0^6 - 1520W_0^4W_1 - 334W_0^2W_1^2 - 8W_1^3 + W_1W_2^2 + 10W_0^2W_3^2 +
  W_1W_3^2 + 1728W_4W_5   \\[.5em]
  &- 4320W_0^3W_6 - 1188W_0W_1W_6 + 20W_0^3W_7 +
  4W_0W_1W_7, \\[.5em]
  0=& -32W_0^4W_2 - 256W_0^2W_1W_2 - 34W_1^2W_2 + 2W_2^3 +
  16W_0^4W_3 + 176W_0^2W_1W_3 + 28W_1^2W_3 - W_3^3 
    \\[.5em]
  &+ 1728W_0^3W_4 +
  1080W_0W_1W_4 + 2048W_0^3W_5 + 512W_0W_1W_5 - 648W_0W_2W_6 - W_0W_2W_7
    \\[.5em]
  & -
  2W_0W_3W_7 + 54W_4W_7 + 16W_5W_7, \\[.5em]
  0=& 32W_0^4W_2 + 316W_0^2W_1W_2 -
  32W_1^2W_2 + 2W_2^3 + 8W_0^2W_1W_3 + 20W_1^2W_3 - 432W_0W_1W_4
    \\[.5em]
  &
   -
  192W_0W_1W_5 + 1080W_0W_2W_6 + 3456W_5W_6 - 9W_0W_2W_7 + 2W_0W_3W_7 +
  54W_4W_7, \\[.5em]
  0=& -16W_0^4W_2 - 142W_0^2W_1W_2 - 8W_1^2W_2 + W_2^3 - 8W_0^2W_1W_3 +
  2W_1^2W_3 + 1836W_0W_1W_4 
    \\[.5em]
  &
  - 256W_0^3W_5 - 128W_0W_1W_5 - 324W_0W_2W_6 +
  5832W_4W_6
  \\
  \hline
\end{array}
$$
\label{CSC3}
\end{table}

\section{Constructing 21-fold cover of CS surface}\label{sec.Z2}
To construct the Cartwright-Steger surface, we needed to understand the surface $Z_2=\Bb^2/(\Ker(\rho_3)\cap \Pi_2)$.
This surface has invariants
$$
h^{0,1}(Z_2)=7,~h^{0,2}(Z_2)= 27,~K_{Z_2}^2 = 189.
$$
Its automorphism group contains $G\times C_3$ where the quotient by the cyclic group of order three gives $\Bb^2/\Ker(\rho_3)$.
This means, in particular, that we may assume that the $C_3$-invariant part of $H^0(Z_2,K_{Z_2})$ is the pullback\footnote{While there is no map $Z_2\to Z_3$, a rational map suffices here.} of $H^0(Z_3,K_{Z_3})$, whose character we will call $(\chi_2+\chi_4)\otimes 1$. 

The generator of $C_3$ has $63$ fixed points of type $(\frac 13,\frac 13)$. This allows us to calculate the dimensions of the $C_3$-character subspaces of $H^i(Z_2,nK_{Z_2})$ by the Holomorphic Lefschetz Formula on holomorphic vector bundles, or Woods Hole Fixed Point Theorem,  cf. \cite{AB}.  Namely, the alternating sum of traces of the action of the generator of $C_3$ on $H^*(Z_2,\Oo_{Z_2})$ is given by 
$ \frac {63}{(1-w)^2} $
where $w=e^{\frac 23 \pi \ii}$. It equals $\frac {63}{(1-w^2)^2} $ for the square of the generator, so the alternating sums of the dimensions of the character components are
$
\frac 13\Big(21 +  \frac {63}{(1-w)^2} +\frac {63}{(1-w^2)^2}\Big) = 14,$
$\frac 13\Big(21 + w^2 \frac {63}{(1-w)^2} +w\frac {63}{(1-w^2)^2}\Big)= 14$
and $\frac 13\Big(21 + w \frac {63}{(1-w)^2} +w^2\frac {63}{(1-w^2)^2}\Big)=-7$ for the characters $1$, $w$ and $w^2$ respectively. We know the dimensions of the character $1$ components, because they equal those of  $Z_3$. This allows us to prove that
$$
\chi_{H^1(Z_2,K_{Z_2})} = V_7 \otimes w^2, ~ \chi_{H^0(Z_2,K_{Z_2})} =  (\chi_2+\chi_4)\otimes 1+ V_{14}\otimes w
$$
where $V_7$ and $V_{14}$ are some representations of $G$ of dimension $7$ and $14$ respectively, and $w$ and $w^2$ denote the characters of $C_3$. We also recall that there is a subgroup 
$G_{21}\subset G\times C_3$ of order $21$ that acts freely on $Z_2$ such that the Cartwright-Steger surface $Z$ is given by
$$
Z=Z_2/G_{21}.
$$
One can see that the only option is for $G_{21}$ to map trivially to $C_3$. Moreover, the structure of $G$ ensures that the order $3$ elements of $G_{21}$ lie in the conjugacy class $3b$ of Table
\ref{gaptable}. The fact that these elements act with no fixed points allows us to determine that $V_{14}=\chi_6$ and $V_7$ is one of $\chi_1+\chi_2$, $\chi_3$, $\chi_4$, $\chi_5$. \footnote{We believe that  $\Lambda^2 H^0(Z_2,\Omega_{Z_2}) \to H^0(Z_2,K_{Z_2})$ is surjective, which leads to $V_7=\chi_3$, but we will not use it in the paper.}

\begin{remark}
As will be more apparent in a moment, it is difficult to exclude all possibilities in this situation. Therefore, our approach is to find the most plausible scenario, with the eventual goal of successfully verifying the final answer to have the correct numerical invariants.
\end{remark}

Now that we know that 
$$
\chi_{H^0(Z_2,K_{Z_2}) }= (\chi_2+\chi_4)\otimes 1+ \chi_6\otimes w,
$$
we will denote the corresponding $27$ variables by
$$
(x_1:\ldots :x_6 :y_1:\ldots :y_7:z_1:\ldots:z_{14}) 
$$
in a way compatible with Section \ref{sec.Z3}. The action of the generators of the group $G$ on the $z$-variables is given in Table \ref{table.actionZ}.
The action of the generators of $G_{21}$ is given in Table \ref{G21Z}.
We are using the same generators as in Tables \ref{table.action} and \ref{G21XY}.

\begin{table}[tbh]
\caption{Action of $G$ on $(z_1,\ldots, z_{14})$}
$$
\begin{array}{|ll|}
\\[-3em]\hline
(z_1,\ldots,z_{14})&\mapsto (z_1, z_4 - z_5 + z_6, z_1 - z_3, z_2 - z_5 + z_6, z_1 - z_5, z_1 - z_6, \\&
 z_{12} - z_4 + z_5 - z_6, -z_{11} + z_{13}, -z_1 + z_{10} - z_2 + z_3 + z_5 + z_9, \\& -z_{10} +
   z_2 + z_4 - z_5 + z_6, -z_1 - z_{11}, z_2 + z_7, -z_1 - z_{11} + z_8, \\&
 z_{10} + z_{11} + z_{14} - z_2 + z_3)
 \\[.5em]
(z_1,\ldots,z_{14})&\mapsto  ( z_2 + z_3, -z_2 - z_7 + z_9, z_1 + z_2, -z_2 - z_5 - z_7 - z_8, -z_7,  \\&
z_2 + z_6, 
 z_{10} + z_7, -z_{12}, z_{11} + z_7, 
 z_{14} - z_2 - z_7, -z_{12} - z_2 + z_4,  \\&
 -z_8, -z_{12} - z_2 + z_4 - z_5, 
 z_{12} + z_{13} - z_4)
 \\
 \hline
\end{array}
$$
\label{table.actionZ}
\end{table}

\begin{table}[tbh]
\caption{Action of $G_{21}$ on $(z_1,\ldots, z_{14})$}
$$
\begin{array}{|ll|}
\\[-3em]\hline
(z_1,\ldots,z_{14})&\mapsto 
(z_2 + z_3, -z_5 + z_6 - z_8, -z_1 + z_3, z_6 + z_9, z_2 + z_3 + z_7, z_3 
 \\&
 - z_6, z_5 - z_6, -z_5, z_1 + z_{11} + z_{14} - z_3 - z_9, -z_{14} - z_5 + z_6 - z_8 
  \\&
 + z_9, z_{12} - z_3 - z_4, z_{10} - z_2 + z_9, -z_3 - z_4, z_1 + z_{13} + z_{14} - z_9)
 \\[.5em]
(z_1,\ldots,z_{14})&\mapsto (z_{11} + z_{14} - z_3 - z_4 + z_6 - z_9, -z_1 - z_{10} - z_{13} - z_{14} + z_4
 \\&
  + z_5 - z_7 +
   z_9, z_{12} + z_{14} - z_3 - z_4 + z_5 + z_8 - z_9, -z_{10} - z_{13} 
    \\&
    - z_{14} + z_2 + 
  z_4 + z_9, z_{11} - z_{13} - z_3 + z_6, -z_3 - z_4 - z_9, 
 z_1 + z_{10} 
  \\&
  + z_{13} - z_4 - z_5 - z_8, -z_1 - z_{11} + z_5 + z_8, 
 z_{13} + z_3 - z_5 - z_8, -z_1
  \\&
   - z_{10} - z_{11} - z_{13} - z_{14} + z_4 + z_5 - z_7 + z_8 +
   z_9, -z_{11} + z_4 + z_8, 
    \\&
    -z_{11} - z_6 + z_8, -z_{11} - z_{14} + z_4 + z_9, 
 z_{11} + z_3 - z_5 - 2 z_8)
 \\
 \hline
\end{array}
$$
\label{G21Z}
\end{table}

We now study the quadratic relations on $(x_1,\ldots,z_{14})$. The relations among $(x_1,\ldots,y_7)$ are well understood from Section \ref{sec.Z3}.
As before, the Holomorphic Lefschetz Formula and Kodaira Vanishing Theorem show that the dimension of $w$ and $w^2$ {eigenspaces} of $H^0(Z_2,2K_{Z_2})$ are $77$ each.
For example, we have 
$$
\frac 13\Big(210 + w^2 \frac  {63 w}{(1 - w)^2} + w \frac {63 w^2}{(1 - w^2)^2} \Big)=77.
$$
We have 
$$
\Sym^2(\chi_6) = \chi_1 + \chi_6 + \chi_7 + \chi_8 + \chi_9 + \chi_{10}
$$
and we make a \emph{guess} that the space of relations has the expected dimension $105-77=28$. This leads to $\chi_1+\chi_{10}$ as relations. Magma calculation shows
that the corresponding equations on $(z_1:\ldots:z_{14})$ cut out a dimension $5$ subscheme in $\Cc\Pp^{13}$. 

\smallskip
The next key idea comes from looking at the relations on $\Big((\chi_2 +\chi_4)\otimes 1 \Big)\otimes \Big(\chi_6\otimes w\Big)$ which are linear combinations of $x_iz_j$ and $y_iz_j$.
These are harder to guess but we only needed a partial guess. We have
$$
\chi_2 \chi_6 = \chi_2 + \chi_4 + \chi_5 + \chi_{13} + \chi_{14} ,~~
\chi_4 \chi_6 = \chi_2 + \chi_{11} + \chi_{13} + \chi_{14 }
$$
and we need to find a space of relations of dimension at least $(6+7)14 - 77=105$. It appeared reasonable to expect some kind of diagonal relation of the character type $\chi_2$.
However,  in contrast to Section \ref{sec.Z3} we can no longer pick this diagonal element randomly! We calculated to high precision several points of $Z_3$ (i.e. values of $x_i$ and $y_j$)
and then used Mathematica to find the diagonal relation of type $\chi_2$ which was a linear combination of $x_iz_j$ and $y_iz_j$ such that the quadratic relations on $z_k$ have a solution.
There were two such diagonal relations, which differ by a sign. We get to pick one of them, which breaks the additional automorphism of Remark \ref{rem.involution}.

\smallskip
Note that given $(x_1,\ldots,y_7)$, the above procedure determines $(z_1,\ldots, z_{14})$ uniquely up to scaling. To get the desired (birational) triple cover we also impose the additional condition
$$
{\rm Cubic}_y= {\rm Cubic}_z
$$
where ${\rm Cubic}_y$ and ${\rm Cubic}_z$ are the $G$-invariant cubic equations on $y_i$ and $z_j$ respectively. 
They are unique up to scaling and we are free to make a choice. The two are proportional because they represent non-zero invariant sections of $3K$ on $Z_3$; that space can be seen to be one-dimensional by a character calculation.

\smallskip
As the result, we were able to construct, with very high precision, multiple points in $\Cc\Pp^{26}$ which come from $Z_2$. We don't claim that $Z_2$ is embedded into $\Cc\Pp^{26}$ by the global sections of the canonical divisor, although we consider it likely.

\begin{remark}
We can then use these points to determine the quadratic relations on 
$x_i,y_j,z_k$. We have found the space of these relations to be of expected dimension $105$. We expect that some degree three equations would also be needed to generate the homogeneous ideal of $Z_2$. However, since our focus lies in constructing the CS surface, we are not particularly interested in $Z_2$ itself. 
\end{remark}

\section{Constructing CS surface}\label{sec.Z}
To understand how the $CS$ surface $Z$ is constructed from $Z_2$, recall that it is given as $Z_2/G_{21}$. The group $G_{21}$ was found explicitly. Then we found $G_{21}$-invariant quadratic polynomials in the $27$ variables $x_i,y_j,z_k$. These form a dimension ten subspace which splits according to the number of $z_k$ into spaces of dimension $4$, $3$ and $3$ for no $z$, one $z$ and two $z$'s. This is precisely the decomposition of $H^0(Z,K_Z)$ into character subspaces of $C_3$ action. 

\smallskip
We picked a basis $U_0,\ldots,U_{9}$ and then we use the numerically constructed points of $Z_2$ to find the polynomial relations on $U_i$.
Specifically, we solve for quadratic and then cubic equations in $U_i$ which are numerically zero on the  points of $Z_2$. This gives approximate linear relations on the coefficients, which are then solved by Mathematica. Since all computations are performed with interval arithmetic, it is possible to calculate the ranks of the matrices in question correctly. At the end of the day, we are able to recognize the coefficients as algebraic numbers.

\begin{remark}\label{mistake}
It is not at all a trivial matter to find a basis $\{U_i\}$ in which relations look nice. Remarkably, after making several informed choices, described below, we were able to find a basis in which the relations have \emph{rational} coefficients! This means that, contrary to expectations, the Cartwright-Steger surface is unique, rather than a pair of surfaces that differ by complex conjugation. The expectation comes from the fact that for fake projective planes, complex conjugation gives rise to holomorphically different surfaces as proved in \cite{KK}.  The paper \cite{Y2} gives an attempt to prove such an expectation in page 1146, where an error was made stating that the oriented cover of $\Rr\Pp^2\sharp T^2$ was
$T^2$. Our results here disprove such an expectation.
\end{remark}

\begin{remark}
We unexpectedly found a quadratic relation on $U_i$. This means that the map
$$
\Sym^2(H^0(Z,2K_Z))\to H^0(Z,4K_Z)
$$
of spaces of dimension $55$ each has a non-trivial kernel. There is an $84$-dimensional space of cubic relations given by the kernel of 
$$\Sym^3(H^0(Z,2K_Z))\to H^0(Z,6K_Z),
$$ as one would expect from the dimension count.
\end{remark}

\begin{remark}
We explain how the choices were made.  
The action of $C_3$ on $Z$ has $9$ fixed points. We picked the basis elements $U_7,U_8,U_9$ so that the $(\frac 13,\frac 13)$ fixed points are $(0:\ldots:0:0:1)$, $(0:\ldots:0:1:0)$, and $(0:\ldots:1:0:0)$, which determines them uniquely up to scaling. The $U_4,U_5,U_6$ are determined by looking at the tangent spaces at these points. The sections $U_0,U_1,U_2,U_3$ are $W_0^2,W_1,W_2,W_3$ from Section \ref{sec.CSC3}, up to scaling. These form the invariant subspace of the $C_3$ action. The other two eigenspaces of the action are $\rm{Span}(U_4,U_5,U_6)$ and $\rm{Span}(U_7,U_8,U_9)$.
\end{remark}

\begin{remark}
We have not been able to reduce the equations of CS surface $Z$ to a form suitable for a printed version, since the coefficients are often about ten decimal digits long (which is better than the initial attempts that led to $40$ digits long Gaussian integers). We make the equations available at \cite{LBhome}.
\end{remark}

\section{Verification}\label{sec.verify}
In this section we comment on the collection of facts that we verified by computer calculations that imply that the surface we constructed is indeed the Cartwright-Steger surface. We have to be fairly creative, since many desirable computations take too long to finish.

\smallskip
We begin by observing that the scheme cut out by the 84 cubic equations is connected and smooth of dimension two. Connectedness follows from a Magma calculation
that the quotient ring is an integral domain, which was done with rational coefficients.
To prove smoothness, we calculated minors of Jacobian matrices of appropriately chosen 
subsets of the equation set (a finite field calculation suffices by semicontinuity). 
Specifically, we looked at the fixed points of $C_3$ action on $Z$ and at each of them found $7$ relations whose gradients are linearly independent at 
that point. Then we considered the size $7$ minors of the Jacobian matrices of these sets of equations. Magma calculation shows that they have no common zeros on $Z$. 

\smallskip
Thus we have a smooth surface $Z$ \footnote{We use the same name as for the CS surface, but we are still to prove that $Z$ is one.} which is embedded into $\Pp^9$ by global sections of a divisor $D$. By Hilbert polynomial calculation and Riemann-Roch theorem,
we see that 
\begin{equation}\label{DZ}
D^2=36,~K_Z D = 18,~\chi(\Oo_Z) = 1.
\end{equation}

\smallskip
The equation $U_0=0$ gives a divisor $2C$ where $C$ is the curve\footnote{We expected that $C$ is the zero locus $E_3$ of the section of the canonical divisor on $Z$. Our calculations below verify it.} whose ideal in $\Cc\Pp^9$ 
is generated by  $U_0$ and eighteen quadratic polynomials in Table \ref{tableU0} below.
\begin{table}[tbh]
\caption{Equations of the $C\subset Z$}
$$\small
\begin{array}{|l|}
\\[-3em]
\hline
U_0, -8 U_1^2 - 2 U_1 U_2 + 2 U_2^2 - 2 U_1 U_3 + 3 U_4 U_9 + 6 U_6 U_9,
 16 U_1^2 + 4 U_2^2 - 4 U_3^2 
 + 3 U_6 U_7, 
 \\
 20 U_1^2 + 12 U_1 U_2 
 + 8 U_1 U_3 + 2 U_3^2 +
  U_5 U_9,
  -4 U_1^2 - 8 U_1 U_2 + 6 U_2^2
   - 4 U_1 U_3 + U_5 U_8,
   \\
 4 U_1^2 + 2 U_1 U_2 - 2 U_2^2 - 2 U_1 U_3 + 2 U_3^2 
 + 3 U_4 U_8,
  -2 U_1^2 + U_2 U_3,
  \\
 -6 U_1 U_4 - 6 U_3 U_4 
 - U_3 U_5 + 12 U_1 U_6 - 12 U_2 U_6 + 6 U_8 U_9,
\\ -36 U_1 U_4 - 12 U_3 U_4 
 - 4 U_1 U_5 - 2 U_2 U_5
 - 4 U_3 U_5 + 24 U_1 U_6 + 3 U_7 U_9,
 \\
 36 U_1 U_4 + 2 U_2 U_5 + 4 U_3 U_5 - 24 U_1 U_6 - 24 U_2 U_6 
 + 3 U_7 U_8,
 \\
 6 U_1 U_4 + 6 U_3 U_4 + U_3 U_5 + 12 U_3 U_6, 
 6 U_1 U_4 + 3 U_2 U_4 + U_1 U_5 + 12 U_1 U_6,
 \\
 18 U_6^2 - 2 U_1 U_8
  + 3 U_2 U_8 - 6 U_1 U_9 + 4 U_2 U_9 + 3 U_3 U_9,
 U_5 U_6 + 2 U_1 U_8 + 2 U_1 U_9 + 2 U_2 U_9,
 \\
 U_5^2 + 12 U_1 U_7 + 3 U_2 U_7 
 + 6 U_3 U_7 -
  16 U_1 U_8 + 48 U_1 U_9 - 4 U_2 U_9,
  \\ 9 U_4 U_6 + 4 U_1 U_8 + 6 U_1 U_9 - 5 U_2 U_9 -
  3 U_3 U_9,
  \\
   3 U_4 U_5 - 3 U_1 U_7 
  - 3 U_3 U_7 - 8 U_1 U_8 - 36 U_1 U_9 + 4 U_2 U_9,
  \\
 18 U_4^2 + 3 U_3 U_7 - 16 U_1 U_8 + 8 U_2 U_9 + 12 U_3 U_9,
 2 U_1 U_8 + U_3 U_8 
 + 2 U_1 U_9 - U_2 U_9
 \\\hline
\end{array}
$$
\label{tableU0}
\end{table}

These equations were found using Mathematica. Once found, we can verify inclusion of ideals by Hilbert polynomial computations, which are readily performed in Magma 
with rational coefficients. We verified that $C$ is an integral (but perhaps singular) curve by verifying that the quotient by its ideal is a domain. 

\smallskip
We now consider the short exact sequence of sheaves 
\begin{equation}\label{ses}
0\to \Oo_Z(-C)(1) \to \OO_Z (1) \to (j_C)_*\Oo_C(1)\to 0.
\end{equation}
We verified by Macaulay2 that $\dim H^1( Z,(j_C)_*\Oo_C(1)) = 1$ while working with rational coefficients.
We have also verified that $H^1(Z,\OO_Z(1))=0$ by a Macaulay2 calculation modulo $101$ (calculations over rationals are more complicated
and appear to be outside of the software and hardware capabilities), which suffices by semicontinuity. 
The long exact sequence in cohomology coming from \eqref{ses} then implies that $h^2(Z,\Oo(-C)(1)) > 0$. Since $2C=D$, this means
$h^2(Z,\Oo(C))>0$, so $H^0(Z,\Oo(K_Z-C))$ is non-zero.
From $2C=D$ and \eqref{DZ}, we see that $CD = K_ZD=18$. Since $D$ is (very) ample, we see that $K_Z-C$ is a degree zero effective divisor.
Therefore, we must have $K_Z-C=0$.
Then we get $K_Z^2=9$, $h^0(K_Z)\geq 1$, which means that $Z$ is a Cartwright-Steger surface. We also see that $D=2C=2K_Z$, so the embedding
is given by $2K_Z$.

\smallskip
\begin{remark}\label{gap}
It is stated that $2K_Z$ is very ample in \cite{Y3}, but the details in the elimination of a potential curve $B$ satisfying Proposition~1-(i) in \cite{Y3} are overlooked, 
as kindly informed by JongHae Keum to the second author.  The
current paper gives a stronger version of the statement.
\end{remark}


\end{document}